\documentclass{amsart}
\usepackage{epic}
\usepackage[all]{xy}
\usepackage{graphicx}

\newtheorem{theorem}{Theorem}[section]
\newtheorem{lemma}[theorem]{Lemma}

\theoremstyle{definition}

\newtheorem{remark}[theorem]{Remark}

\renewcommand{\P}{{\mathbb{P}}}
\newcommand{\Z}{{\mathbb{Z}}}
\newcommand{\C}{{\mathbb{C}}}
\newcommand{\Sb}{{\mathbb{S}}}
\newcommand{\Fc}{{\mathcal{F}}}
\newcommand{\Gc}{{\mathcal{G}}}
\newcommand{\Oc}{{\mathcal{O}}}

\newcommand{\ed}{\widehat{E}}

\newcommand\res[1]{{\lower1pt\hbox{$|$}}_{\raise.5pt\hbox{${\scriptstyle #1}$}}}

\numberwithin{equation}{section}

\begin{document}

\title{Tate Resolutions for Segre Embeddings}

\author{David A.\ Cox}
\address{Department of Mathematics and Computer Science, Amherst
College, Amherst, MA 01002-5000, USA}
\email{dac@cs.amherst.edu}

\author{Evgeny Materov}
\address{Department of Mathematics and Statistics, University of
  Massachusetts, Amherst, MA 01003-9305, USA}
\email{materov@math.umass.edu}


\begin{abstract}
We give an explicit description of the terms and differentials of the
Tate resolution of sheaves arising from Segre embeddings of
$\P^a\times\P^b$.  We prove that the maps in this Tate resolution are
either coming from Sylvester-type maps, or from Bezout-type maps
arising from the so-called toric Jacobian. 
\end{abstract}

\date{\today}

\maketitle

\section{Introduction}

Let $V$ and $W$ be dual vector spaces of dimension $N + 1$ over a
field $K$ of characteristic $0$.  It is known that there is a relation
between complexes of free graded modules over the exterior algebra
$E = \bigwedge V$ and coherent sheaves on projective space $\P(W)$.
More precisely, the Bernstein-Gel$\!$\'{}$\!$fand-Gel$\!$\'{}$\!$fand
(BGG) correspondence introduced in \cite{BGG} establishes an
equivalence between the derived category of bounded complexes of
coherent sheaves on $\P(W)$ and the stable category of complexes of
finitely generated graded modules over $E$.  The essential part of
this correspondence is given via the \emph{Tate resolutions}, namely
for any coherent sheaf $\Fc$ on $\P(W)$ there exists a bi-infinite
exact sequence
\[
  T^{\bullet}(\Fc) :\,
  \cdots  \longrightarrow
  T^{-1}(\Fc) \longrightarrow T^{0}(\Fc) \longrightarrow
  T^{1}(\Fc)  \longrightarrow \cdots
\]
of free graded $E$-modules.  The terms of Tate resolution were
described explicitly by Eisenbud, Fl{\o}ystad and Schreyer \cite{EFS}
in the form
\begin{eqnarray*}
  T^p (\Fc) = {\textstyle \bigoplus_{i}}\,
  \ed(i-p)\otimes H^i(\P(W),\Fc(p-i)),
\end{eqnarray*}
where $\ed = \omega_E = \mathrm{Hom}_K(E,K)= \bigwedge W$ as an
$E$-module.

While the terms of Tate resolutions are described explicitly, the maps
are much more difficult to describe. The knowledge of the maps give
us, for example, an opportunity to compute generalized resultants
(see, e.g., \cite{ES} or \cite{khetan1, khetan2}).

In \cite{Cox_bez} Cox found an explicit construction of the Tate
resolution for the $d$-fold Veronese embedding
\[
\nu_d : \P^n\rightarrow \P^{\binom{n+d}{d}-1}
\]
of $\P^n$ when $\Fc = \nu_{d*} \Oc_{\P^n}(k)$. The construction of
differentials in Tate resolution involves the Bezoutian of $n+1$
homogeneous polynomials of degree $d$ in $n+1$ variables. In this
paper, we find a similar description of the Tate resolution arising
from the Segre embedding
\[
  \nu:\P^a\times\P^b\rightarrow \P^{ab + a + b}
\]
of the sheaf $\nu_{*} \Oc_{\P^a\times\P^b}(k,l)$. The shape of the
Tate resolution depends only on the pair $(k,l)$ and there are three
types of possible resolutions:
\begin{align*}
  &\text{Type 1:} \ -a \le k - l \le b\\
  &\text{Type 2:} \ k - l > b\\
  &\text{Type 3:} \  k - l < -a.
\end{align*}
We prove that Type 1 maps involve the toric Jacobian of a sequence
bilinear forms $f_0,\ldots,f_{a+b}$ in $x_0,\ldots,x_a$,
$y_0,\ldots,y_b$ given by
\[
f_j = \sum_{i,k} a_{ijk}\, x_i\, y_k,\quad 0\le j\le a+b.
\]
This result resembles the Bezout-type formulas for hyperdeterminants
of a three-dimensional matrix $A = (a_{ijk})$ discussed in
\cite[Chapter 14, Theorem 3.19]{GKZ}.  The resolutions of Type 2 and 3
are similar to each other and both arise from the Sylvester forms of
$f_0,\ldots,f_{a+b}$.  Notice that similar formulas appear in the
study of Bondal type formulas for hyperdeterminants of $A$ (see
\cite[Chapter 14, Theorem 3.18]{GKZ}).

The situations considered in this paper and \cite{Cox_bez} are special
cases when $\Fc$ is a push-forward of ${\mathcal L} =
\Oc(m_1,\ldots,m_r)$ in the projective embedding
\[
\nu : \P^{l_1}\times\cdots\times\P^{l_r}\rightarrow 
\P(S^{d_1} K^{l_1 + 1}\otimes\cdots\otimes S^{d_r} K^{l_r + 1})
\]
which is a combination of Segre and Veronese embeddings. This case
will be studied in a different paper \cite{Cox_Mat}.  We conjecture
that the maps in the Tate resolutions are essentially the same as in
Weyman-Zelevinsky complexes studied in \cite{WZ} or the same as in the
resultant spectral sequences from \cite[Chapter 4, Section 3]{GKZ}.

Here is the outline of our paper.  In Section~2 we give a definition
of the Tate resolution and explain its basic properties.  The main
parts of the paper are Sections~3 and 4. In Section~3 we describe the
terms of Tate resolution arising from Segre embeddings of sheaves on
products of projective spaces, and in Section~4 we find explicit forms
for corresponding differentials.

\section{Basic Definitions and Properties of Tate Resolutions}

\subsection{Graded exterior algebras}
Given $V$ and $W$ as above, the algebras $S = \mathrm{Sym}\,W$ and $E
= \bigwedge V$ are graded by the following convention: $\deg(e_i) = 1$
for a basis $e_0, e_1,\ldots, e_N$ of $W$ and $\deg(e_i^*) = -1$ for
the dual basis $e_0^*, e_1^*,\ldots, e_N^*$ of $V$, so that $E_{-i} =
\bigwedge^i V$.  Define $E(p)$ as the graded $E$-module with $E(p)_q =
E_{p+q}$.  Then any free graded $E$-module is an $E$-module of the
form
\[
  M = {\textstyle \bigoplus_i}\, E(-i)\otimes V_i,
\]
where $V_i$ is a finite dimensional $K$-vector space with $V_i = \{0
\}$ for almost all $i$.  Note that $V_i$ gives the degree $i$ generators
of $M$, because $(E(-i)\otimes V_i)_i = E(-i)_i\otimes V_i =
E_0\otimes V_i = V_i$.

The dual to $E$ algebra $\widehat{E} = \omega_E = \mathrm{Hom}_K(E,
K)$ is a left $E$-module with the graded components $\widehat{E}_i =
\mathrm{Hom}_K(E_{-i}, K) = \mathrm{Hom}_K(\bigwedge^i V, K)$.  The
perfect pairing
\[
  {\textstyle\bigwedge}^i V \times{\textstyle\bigwedge}^i W
  \longrightarrow K
\]
implies $\widehat{E}_i = \bigwedge^i W$ and $\widehat{E} = \bigwedge
W$.  Moreover, $\widehat{E}$ is Gorenstein, i.e., $\widehat{E}$ is
isomorphic to $E$ with a shift in grading.  Namely, the isomorphism
$\bigwedge^i V\otimes\bigwedge^{N+1} W\rightarrow
\bigwedge^{N+1-i} W$ implies
\[
  \widehat{E} = E(-N-1)\otimes {\textstyle\bigwedge}^{N+1}W,
\]
and therefore $\widehat{E}\cong E(-N-1)$ (non-canonically) via a map
$\bigwedge^{N+1}W\cong K$.  For later purposes, we note the canonical
isomorphism
\begin{equation}
\label{canonical}
  \mathrm{Hom}_E(\widehat{E}(p)\otimes A,\widehat{E}(q)\otimes B)_0
  \simeq \mathrm{Hom}_K({\textstyle\bigwedge}^{p-q}W \otimes A,B),
\end{equation}
where the subscript ``$0$'' denotes graded homomorphisms of degree zero.

\subsection{Tate resolutions}
\label{tateres}
By \cite{EFS} or \cite{Floystad_KD}) a coherent sheaf $\Fc$ on $\P(W)$
determines a \emph{Tate resolution} $T^\bullet(\Fc)$, which is an
(unbounded) acyclic complex
\[
  T^{\bullet}(\Fc) :\,
  \cdots  \longrightarrow
  T^{-1}(\Fc) \longrightarrow T^{0}(\Fc) \longrightarrow
  T^{1}(\Fc)  \longrightarrow \cdots
\]
of free graded $E$-modules with the terms
\[
  T^p(\Fc) = {\textstyle \bigoplus_i}\,
  \widehat{E}(i-p)\otimes
  H^i(\P(W), \Fc(p-i)).
\]

For example, in degree $k$ we have
\begin{equation}
\label{graded_terms}
  T^p(\Fc)_k =
  {\textstyle \bigoplus_i}\,{\textstyle\bigwedge}^{i-p+k}W\otimes
  H^i(\P(W), \Fc(p-i))
\end{equation}
since $\widehat{E}(i-p)_k = \widehat{E}_{i-p+k} = \bigwedge^{i-p+k}
W$.  The Tate resolution is defined by each differential
$d^p:T^p(\Fc)\rightarrow T^{p+1}(\Fc)$ since $T^{\ge p}(\Fc)$ is a
minimal injective resolution of $\mathrm{ker}(d^p)$ and $T^{<p}(\Fc)$
is a minimal projective resolution of $\mathrm{ker}(d^p)$
\cite{Eisenbud_syz}.

When the context is clear, we will write $H^i(\Fc(j))$ instead of
$H^i(\P(W),\Fc(j))$.

\begin{lemma} For fixed $k$, $T^p(\Fc)_k= 0$ if either
$p > k + m$ or $p < k -N - 1$, where $m = \dim(\mathrm{supp}(\Fc))$.
\end{lemma}

\begin{proof} Since $H^i(\Fc(p-i)) = 0$ if $i<0$ or $i> m$, we may
assume $0\le i\le m$.  Then the inequalities $k +m < p$, $i\le m$
easily imply
\[
  i - p + k \le m - p + k < -p+p = 0,
\]
so that $\bigwedge^{i - p + k} W = 0$.  Analogously, if $k -N - 1 >
p$, $i \ge 0$, then
\[
  i - p + k \ge -p + k > -p + p + N + 1 = N+1,
\]
so that we again have $\bigwedge^{i - p + k} W = 0$.
\end{proof}

\begin{lemma}
\label{map_ij}
If $i < j$, then the map
\[
  d^p_{i,j}:
  \widehat{E}(i-p)\otimes H^i(\Fc(p-i))\longrightarrow
  \widehat{E}(j-p-1)\otimes H^j(\Fc(p+1-j))
\]
from the $i$th summand of $T^p(\Fc)$ to the $j$th summand of
$T^{p+1}(\Fc)$ is zero.
\end{lemma}

\begin{proof}
Let $A = H^i(\Fc(p-i))$ and $B = H^j(\Fc(p+1-j))$.  By
\eqref{canonical}, $d^p_{i,j}$ lies in
\[
  \mathrm{Hom}_E(\widehat{E}(i-p)\otimes A,
  \widehat{E}(j-p-1)\otimes B)_0 \simeq
  \mathrm{Hom}_K({\textstyle\bigwedge}^{i-j+1}W\otimes A,B).
\]
It follows that $d^p_{i,j} = 0$ when $i+1 < j$ and that $d^p_{i,i+1}$
is constant.  Then minimality implies that $d^p_{i,i+1} = 0$.
\end{proof}

Finding an explicit expression for differentials
$d^p:T^{p}(\Fc)\rightarrow T^{p+1}(\Fc)$ seems to be a difficult
problem.  By Lemma~\ref{map_ij}, the general maps from the $i$th
summand of $T^{p}(\Fc)$ in the Tate resolution $T^\bullet(\Fc)$ have
the form
\[
  \widehat{E}(i-p)\otimes H^i(\Fc(p-i))\longrightarrow
  {\textstyle\bigoplus_{j\ge 0}}\,
  \widehat{E}(i-j - p-1)\otimes H^{i-j}(\Fc(p+1-i+j)).
\]
The ``horizontal'' component of this map is explicitly known:
\begin{align*}
  \widehat{E}(i-p)\otimes H^i(\Fc(p-i)) &\longrightarrow
  \widehat{E}(i-p-1)\otimes H^i(\Fc(p+1-i)) \\
  f\otimes m &\longmapsto {\textstyle\sum_i}\,
  f e_i^*\otimes e_i m.
\end{align*}
By \eqref{canonical}, this corresponds to the multiplication map
\[
  W\otimes H^i(\Fc(p-i))\longrightarrow H^i(\Fc(p+1-i)).
\]
One of the main results of this paper is an explicit description the
entire differential $d^p :T^{p}(\Fc)\rightarrow T^{p+1}(\Fc)$ in some
special situations.

\section{Tate Resolutions for Segre Embeddings of $\P^a\times \P^b$}
\label{paxpb}

Let $X = \P^a\times\P^b$, with
coordinate ring $S = K[\mathbf{x},\mathbf{y}]$ for variables $\mathbf{x} =
(x_0,\dots,x_a)$, $\mathbf{y} = (y_0,\dots,y_b)$.  The ring $S$ has a
natural bigrading where the $\mathbf{x}$ variables have degree $(1,0)$
and the $\mathbf{y}$ variables have degree $(0,1)$.  The graded piece
of $S$ in degree $s,t$ will be denoted $S_{s,t}$.  Set
\[
  W = H^0(X,\Oc_X(1,1)) = S_{1,1}
\]
and let
\[
  \nu : X = \P^a\times\P^b \longrightarrow \P(W) \simeq \P^{ab+a+b}
\]
be the Segre embedding.  The sheaf
\begin{equation}
\label{defineF}
  \Fc = \nu_* \Oc_X(k,l)
\end{equation}
has Tate resolution $T^{\bullet}(\Fc)$ with
\begin{equation}
\label{tlmp}
\begin{aligned}
  T^p(\Fc) &= {\textstyle\bigoplus_{i}}\,\widehat{E}(i-p)\otimes
  H^i(\Fc(p-i))\\ &=
  {\textstyle\bigoplus_{i}}\,\widehat{E}(i-p)\otimes H^i(X,\Oc_X(k+
  p-i,l+p-i)).
\end{aligned}
\end{equation}
In general, we say that the summand $\widehat{E}(i-p)\otimes
H^i(\Fc(p-i))$ of $T^p(\Fc)$ has \emph{cohomological level} $i$.
Since
\[
  H^i(X,\Oc_X(k+ p-i,l+p-i)) = 0\ \text{for}\ i \notin \{0,a,b,a+b\},
\]
we see that $T^p(\Fc)$ has at most four nonzero cohomological levels.

In Section~\ref{tateres}, we observed that the ``horizontal''
components of the differntials $d^p:T^{p}(\Fc) \to T^{p+1}(\Fc)$ are
explicitly known.  The main result of this paper is a description of
the ``diagonal'' components of these maps.

\subsection{Regularity}

We recall that a coherent sheaf $\Fc$ is called $m$-\emph{regular} if
\[
  H^i(\Fc(m - i)) = 0, \quad \text{for all}\ i>0.
\]
If $\Fc$ is $m$-regular, then it is known that it is also
$(m+1)$-regular.  The \emph{regularity} of $\Fc$, denoted
$\mathrm{reg}(\Fc)$, is the unique integer $m$ such that $\Fc$ is
$m$-regular, but not $(m-1)$-regular.  It follows from
the definition of regularity if $m = \mathrm{reg}(\Fc)$, then 
\[
  T^p(\Fc) = \widehat{E}(-p)\otimes H^0(\Fc(p)),\quad p \ge m,
\]
and the Tate resolution has the form:
\begin{equation*}
  \cdots \longrightarrow
  T^{m-2}(\Fc)\longrightarrow
  T^{m-1}(\Fc)\longrightarrow
  \widehat{E}(-m)\otimes H^0(\Fc(m))
  \longrightarrow
   \cdots.
\end{equation*}

We now compute the regularity of the sheaf $\Fc$ defined in
\eqref{defineF}.

\begin{lemma}
\label{reglemma}
$\mathrm{reg}(\Fc) = \max\big\{\!-\min\{k,l\},\min\{b-k,a-l\}\big\}$.
\end{lemma}

\begin{proof}
Let $m_0$ denote the right-hand side of the above equation and let $m
\ge m_0$.  Then Serre duality implies
\begin{align*}
  H^{a+b}(\Fc(m-(a+b))) &= H^{a+b}(X,\Oc_X(k+m-(a+b),l+m-(a+b)))\\
  &\simeq H^0(X,\Oc_X(b-k-(m+1),a-l-(m+1)))^*.
\end{align*}
Since $m \ge m_0$ implies $m \ge b-k$ or $m \ge a-l$, we see that
$H^{a+b}(\Fc(m-(a+b))) = 0$.

Next we use the K\"unneth formula to write
\begin{align*}
  H^{a}(\Fc(m-a)) &= H^{a}(X,\Oc_X(k+m-a,l+m-a))\\
  &=H^0(\P^a,\Oc(k+m-a))\otimes H^a(\P^b,\Oc(l+m-a))\ \oplus \ \\
  &\quad\, H^a(\P^a,\Oc(k+m-a))\otimes H^0(\P^b,\Oc(l+m-a)).
\end{align*}
Since $m \ge m_0$ implies $m \ge -k$ and $m \ge -l$, we see that $k+m
-a \ge -a$, which implies $H^a(\P^a,\Oc(k+m-a) = 0$.  Furthermore,
$H^a(\P^b,\Oc(l+m-a)) = 0$ when $a \ne b$, and when $a = b$, we have
$l+m-a = l+m-b \ge -b$, which again implies $H^a(\P^b,\Oc(l+m-a)) =
0$.  Hence $H^{a}(\Fc(m-a)) = 0$, and $H^{b}(\Fc(m-b)) = 0$ is proved
similarly.

It follows that $m_0 \ge \mathrm{reg}(\Fc)$.  To prove equality, we
will let $m = m_0-1$ and show that $H^{i}(\Fc(m-i)) \ne 0$ for some $i
> 0$.  We consider two cases.

\emph{Case 1:} $m_0 = \min(b-k,a-l) \ge -\min\{k,l\}$.
This implies the inequalities
$b-k-(m+1)\ge 0$ and $a-l-(m+1)\ge 0$.  Hence
\[
H^{a+b}(\Fc(m-(a+b))) \simeq H^0(X,\Oc_X(b-k-(m+1),a-l-(m+1)))^* \ne
0.
\]

\emph{Case 2:} $m_0 = -\min(k,l) > \min(b-k,a-l)$.  If $m_0 = -k$,
then $k+m-a = -a-1$, so that $H^a(\P^a,\Oc(k+m-a)) \ne 0$.  We also
have $m_0 > \min\{b-k,a-l\}$, so that $m_0 > b-k$ or $m_0 > a-l$.  The
former is impossible since $m_0 = -k$, and then the latter implies
$l+m-a \ge 0$, so that $H^0(\P^b,\Oc(l+m-a)) \ne 0$.  By K\"unneth,
\[
  0 \ne H^a(\P^a,\Oc(k+m-a)) \otimes H^0(\P^b,\Oc(l+m-a)) \subseteq
  H^{a}(\Fc(m-a)).
\]
The proof when $m_0 = -l$ is similar.
\end{proof}

To see what this says about the Tate resolution of $\Fc$, we define
\begin{equation}
\label{ppmdef}
\begin{aligned}
  p^+ &= \max\big\{\!-\min\{k,l\},\min\{b-k,a-l\}\big\}\\
  p^- &= \min\hskip2pt\big\{\!-\min\{k,l\},\min\{b-k,a-l\}\big\} - 1.
\end{aligned}
\end{equation}
Then we have the following result.

\begin{lemma}
\label{ppmlemma}
\[
  T^p(\Fc) = \begin{cases}
  \widehat{E}(-p) \otimes S_{k+p,l+p} & p \ge p^+\\[3pt]
  \widehat{E}(a+b-p) \otimes S_{b-k-1-p,a-l-1-p}^* & p \le p^-.
\end{cases}
\]
\end{lemma}

\begin{proof}
The assertion for $p \ge p^+$ follows immediately from
Lemma~\ref{reglemma} and the discussion preceding the lemma. 
For $p \le p^-$, note that
\begin{align*}
  H^{a+b}(\Fc(p-(a+b))) &\simeq H^0(X,\Oc_X(b-k-(p+1),a-l-(p+1)))^*\\
  &= S_{b-k-1-p,a-k-1-p}^*
\end{align*}
and that
\begin{align*}
  H^{a+b-i}(\Fc(p-(a+b-i))) &\simeq H^i(X,\Oc_X(b-k-1-p-i,a-l-1-p-i))^*\\
  &= H^i(\Gc(-p-i)),
\end{align*}
where $\Gc = \nu_*\Oc_X(b-k-1,a-l-1)$.  Applying Lemma~\ref{reglemma}
to $\Gc$, we see that $H^i(\Gc(-p-i)) = 0$ whenever $i > 0$ and
\[
  -p \ge
   \max\big\{-\min\{b-k-1,a-l-1\},\min\{b-(b-k-1),a-(a-l-1)\}\big\},
\]
which is equivalent to $p \le p^-$.
\end{proof}

Lemma~\ref{ppmlemma} tells us that for $p^-$ and below, the Tate
resolution lives at cohomological level $a+b$, and for $p^+$ and
above, it lives at cohomological level $0$.

\subsection{The Shape of the Resolution}
\label{shape}
For $k,l \in \Z$, the Tate
resolution of $\Fc = \nu_* \Oc_X(k,l)$ on $X = \P^a\times\P^b$ has one
of the following three types:
\begin{align*}
  &\text{Type 1:} \ -a \le k - l \le b\\
  &\text{Type 2:} \ k - l > b\\
  &\text{Type 3:} \  k - l < -a.
\end{align*}
We will prove three lemmas, one for each type.

\begin{lemma}[Type 1]
\label{type1lemma}
Assume that $\Fc$ has Type 1.  Then $p^- = -\min\{k,l\}-1$
and $p^+ =
\min\{b-k,a-l\}$.  Furthermore, if $p^- < p < p^+$, then
\[
  T^p(\Fc) \ = \ \begin{matrix} \widehat{E}(a+b-p)\otimes
  S_{b-k-1-p,a-l-1-p}^*\\[5.5pt]
  \bigoplus\\[3pt]
  \widehat{E}(-p)\otimes S_{k+p,l+p}. \end{matrix}
\]
\end{lemma}

\begin{proof}
Since $a$ and $b$ are positive, the inequality $-a \le k - l \le b$
implies that $-\min\{k,l\} \le \min\{b-k,a-l\}$.  Using
\eqref{ppmdef}, we get the desired formulas for $p^-$ and $p^+$.

Now assume that $p^- < p < p^+$.  Recall that $H^{a}(\Fc(p-a))$ is
isomorphic to
\[
  H^0\!(\P^a\!,\Oc(k+p-a))\otimes\! H^a\!(\P^b\!,\Oc(l+p-a))\oplus
  H^a\!(\P^a\!,\Oc(k+p-a))\otimes\! H^0\!(\P^b\!,\Oc(l+p-a)).
\]
If the second summand is nonzero, then $k+p-a < -a$ and $l+p-a \ge 0$,
which implies $k-l < -a$, a contradiction.  If the first summand is
nonzero, then $a = b$, $k+p-a \ge 0$ and $l+p-a < -a$.  These imply
$k-l > a = b$, again a contradiction.  Hence $H^{a}(\Fc(p-a)) = 0$.
A similar argument shows that $H^{b}(\Fc(p-b)) = 0$.
\end{proof}

Thus, when $\Fc$ has Type 1, the differential $d^p:T^{p}(\Fc) \to
T^{p+1}(\Fc)$ looks like
\[
  \xymatrix@C=20pt@R=10pt{
  \widehat{E}(a+b-p)\!\otimes\!
  S_{b-k-1-p,a-l-1-p}^* \ar[r] \ar[ddr]^{\ d^p_{a+b,0}} &
  \widehat{E}(a+b-p-1)\!\otimes\!
  S_{b-k-p-2,a-l-p-2}^* \\
  \bigoplus & \bigoplus \\
  \widehat{E}(-p)\!\otimes\!
  S_{k+p,l+p} \ar[r] &
  \widehat{E}(-p-1)\!\otimes\!
  S_{k+p+1,l+p+1}
  }
\]
Hence a Type 1 Tate resolution has cohomological levels
$a+b$ (the top row) and $0$ (the bottom row).  Section~\ref{diagonal1}
will discuss $d^p_{a+b,0}$.

\begin{lemma}[Type 2]
\label{type2lemma}
Assume that $\Fc$ has Type 2.  Then $p^- = b-k-1$ and $p^+ = -l$.
Furthermore, if $p^- < p < p^+$, then
\[
  T^p(\Fc) = \widehat{E}(b-p)\otimes S_{k+p-b,0}\otimes S_{0,-l-p-1}^*.
\]
\end{lemma}

\begin{proof}
Since $a$ and $b$ are positive, the inequality $k - l > b$ implies
$\min\{k,l\} = l$, $\min\{b-k,a-l\} = b-k$.  Using $k - l > b$ again,
\eqref{ppmdef} gives the desired formulas for $p^-$ and $p^+$.

Now assume that $p^- < p < p^+$.  Then
\[
  H^{a+b}(\Fc(p-(a+b)) \simeq H^0(X,\Oc_X(b-k-1-p,a-l-1-p)^* = 0
\]
since $p > p^- = b-k-1$.  Furthermore, $p < p^+ = -l$ implies $l+p - b
< 0$, so that
\[
  H^b(\P^a,\Oc(k+p-b))\otimes H^0(\P^b,\Oc(l+p-b)) = 0.
\]
Hence, by K\"unneth and Serre duality on $\P^b$,
\begin{align*}
  H^b(\Fc(p-b)) &\simeq H^b(X,\Oc_X(k+p-b,l+p-b))\\
  &\simeq H^0(\P^a,\Oc(k+p-b))\otimes H^b(\P^b,\Oc(l+p-b))\\
  &\simeq S_{k+p-b,0}\otimes S_{0,-l-p-1}^*.
\end{align*}

Finally, if $a \ne b$, we also have $H^a(\P^b,\Oc(l+p-a)) = 0$, and
$H^0(\P^b,\Oc(l+p-a)) = 0$ also holds since $l+p - a < 0$.  Hence
$H^a(\Fc(p-a)) = 0$ when $a \ne b$.  A similar argument shows $H^0(\mathcal{F}(p)) = 0$.  
\end{proof}

Lemma~\ref{type2lemma} tells us that for Type 2 Tate resolutions, the
only nonzero diagonal maps appear in $T^{p^-}(\Fc) \to T^{p^- +
1}(\Fc)$:
\[
  \xymatrix@C=15pt@R=15pt{
  \widehat{E}(a+1+k)\otimes S_{0,a+k-l-b}^*
  \ar[dr]^(.55){d^{p^-}_{a+b,b}} \\ & \widehat{E}(k)\otimes
  S_{0,0}\otimes S_{0,k-l-b-1}^*
  }
\]
(at cohomological levels $a+b$ and $b$) and in $T^{p^+ - 1}(\Fc) \to
T^{p^+}(\Fc)$:
\[
  \xymatrix@C=15pt@R=15pt{
  \widehat{E}(b+1+l)\otimes S_{k-l-b-1,0}\otimes S_{0,0}^*
  \ar[dr]^(.65){d^{p^+ - 1}_{b,0}}\\
  & \widehat{E}(l)\otimes S_{k-l,0}
  }
\]
(at cohomological levels $b$ and $0$).  The diagonal maps
$d^{p^-}_{a+b,b}$ and $d^{p^+ - 1}_{b,0}$ will be discussed in
Section~\ref{diagonal2}.

\begin{lemma}[Type 3]
\label{type3lemma}
Assume that $\Fc$ has Type 3.  Then $p^- = a-l-1$ and $p^+ = -k$.
Furthermore, if $p^- < p < p^+$, then
\[
  T^p(\Fc) = \widehat{E}(a-p)\otimes S_{-k-p-1,0}^*\otimes S_{0,l+p-a}.
\]
\end{lemma}

\begin{proof}
The proof is similar to the proof of Lemma~\ref{type2lemma} and hence
is omitted.
\end{proof}

Lemma~\ref{type3lemma} tells us that for Type 3 Tate resolutions, the
only nonzero diagonal maps appear in $T^{p^-}(\Fc) \to T^{p^- +
1}(\Fc)$:
\[
  \xymatrix@C=15pt@R=15pt{
  \widehat{E}(b+1+l)\otimes S_{b-k+l-a,0}^*
  \ar[dr]^(.55){d^{p^-}_{a+b,a}} & \\ & \widehat{E}(l)\otimes
  S_{l-k-a-1,0}^* \otimes S_{0,0}
  }
\]
(at cohomological levels $a+b$ and $a$) and  in $T^{p^+ - 1}(\Fc) \to
T^{p^+}(\Fc)$:
\[
  \xymatrix@C=15pt@R=15pt{
  \widehat{E}(a+1+k)\otimes S_{0,0}^*\otimes S_{0,l-k-a-1}
  \ar[dr]^(.65){d^{p^+ - 1}_{a,0}}\\
  & \widehat{E}(k)\otimes S_{0,l-k}
}
\]
(at cohomological levels $a$ and $0$).  The diagonal maps
$d^{p^-}_{a+b,a}$ and $d^{p^+ - 1}_{a,0}$ will be discussed in
Section~\ref{diagonal2}.

\begin{remark}
We finish this section by noting that some of the Tate resolutions
considered here can be found in Fl{\o}ystad's paper
\cite{Floystad_hom}.  Specifically, let $W_1$ and
$W_2$ be finite dimensional $K$-vector spaces, and consider
the Tate resolution associated to $\Fc = \nu_*{\mathcal L}$, where
${\mathcal L}=\Oc_{\P(W_1)\times\P(W_2)}(-2,a)\otimes\wedge^{a+1}
W_1$, $\dim W_1 = a+1$, where
\[
\nu:\P(W_1)\times\P(W_2)\longrightarrow \P(W_1\otimes W_2)
\] 
is the Segre embedding.  The results of our paper apply to this Tate
resolution.  

Now consider a surjective map $W_1^*\otimes W_2^* \rightarrow W^*$.
This gives a projection
\[
\pi: \P(W_1\otimes W_2) \dasharrow \P(W)
\]
whose center is disjoint from the image of the Segre map.  By
\cite[Section 1.2]{Floystad_hom}, the Tate resolution for $\Fc$ gives
a Tate resolution for $\Gc = \pi_*\Fc$.  Fl{\o}ystad shows that this
projected Tate resolution has the form
\[
\cdots \to T^{-1}(\Gc)
\to
T^0(\Gc) = \widehat{E}(a)\otimes W_1^* \to
T^1(\Gc) = \widehat{E}(a-1)\otimes W_2^* \to T^2(\Gc)
\to \cdots 
\]
with the map $d^0:T^0(\Gc)\rightarrow T^1(\Gc)$ coming from the
surjection $W_1^*\otimes W_2^* \rightarrow W^*$ (see \cite[Theorem
2.1]{Floystad_hom}).
\end{remark}

\section{The Maps in the Tate resolution for Segre Embeddings of
$\P^a\times\P^b$} 
\subsection{Type 1 Diagonal Maps}
\label{diagonal1}
We will use the toric Jacobian from \cite[\S4]{Cox_res}.  The fan for
$\P^a\times\P^b$ has $a+b+2$ 1-dimensional cone generators
$e_0,\dots,e_a,e_0',\dots,e_b'$, corresponding to
$x_0,\dots,x_a,y_0,\dots,y_b$.  The generators
$e_1,\dots,e_a,e_0',\dots,e_{b-1}'$ are linearly independent.  Given
$f_0,\dots,f_{a+b} \in S_{1,1}$, the toric Jacobian is
\begin{eqnarray}
\label{toric_Jac}
  J(f_0,\dots,f_{a+b}) = \frac{1}{x_0y_b} \det\begin{pmatrix}
  f_0 & \cdots & f_{a+b}\\
  \frac{\partial f_0}{\partial x_1} & \cdots & \frac{\partial
  f_{a+b}}{\partial x_1}\\
  \vdots & & \vdots\\
  \frac{\partial f_0}{\partial x_a} & \cdots & \frac{\partial
  f_{a+b}}{\partial x_a}\\
  \frac{\partial f_0}{\partial y_0} & \cdots & \frac{\partial
  f_{a+b}}{\partial y_0}\\
  \vdots & & \vdots\\
  \frac{\partial f_0}{\partial y_{b-1}} & \cdots & \frac{\partial
  f_{a+b}}{\partial y_{b-1}}
  \end{pmatrix}.
\end{eqnarray}
Since $f_i \in S_{1,1} = W$, we see that $J(f_0,\dots,f_{a+b}) \in
S_{b,a}$, where $(b,a)$ is the ``critical degree,'' often denoted
$\rho$ in the literature on toric residues.

This toric Jacobian is closely related to the
$(a+1)\times(a+b+1)\times(b+1)$ hyperdeterminant discussed in
\cite[14.3.D]{GKZ}.  The connection becomes especially clear when we
use the graph intepretation from \cite[pp.\ 473--474]{GKZ}.  The idea
is as follows.

Fix distinct monomials $f_0,\dots,f_{a+b} \in S_{1,1}$.  These give a
bipartite graph $G$ with $a+b+2$ vertices $x_0,\dots,x_a,y_0,\dots,y_b$
and $a+b+1$ edges given by the monomials, where $f_\ell = x_iy_j$ is
regarded as the edge connecting $x_i$ to $y_j$.  The incidence matrix
of $G$ is the $(a+b+2)\times(a+b+1)$ matrix whose rows
correspond to vertices and columns correspond to edges, and where an
entry is $1$ is the vertex lies on the edge and is $0$ otherwise.

Let $M$ denote the square matrix obtained from the incidence
matrix by removing the bottom row.  Then we have the following result.

\begin{lemma}
\label{torictree}
Let $f_0,\dots,f_{a+b} \in S_{1,1}$ be distinct monomials and let $M$
be the matrix described above.  Then:
\begin{enumerate}
\item The toric Jacobian of $f_0,\dots,f_{a+b}$ is given by
\[
  J(f_0,\dots,f_{a+b}) = \det M\, \frac{\prod_\ell f_\ell}{\prod_i x_i
  \prod_j y_j}.
\]
\item $\det M \in \{0,\pm1\}$, and $\det M = \pm 1$ if and only if $G$
is a tree.
\end{enumerate}
\end{lemma}

\begin{proof}
Each $f_\ell$ is homogeneous of degree $1$ in $x_0,\dots,x_a$, so
$f_\ell = \sum_i x_i \frac{\partial f_\ell}{\partial x_i}$.  Hence
the toric Jacobian $J(f_0,\dots,f_{a+b})$ can be written
\[
  \frac{1}{x_0y_b} \det\begin{pmatrix}
  x_0\frac{\partial f_0}{\partial x_0} & \!\!\cdots \!\!& x_0\frac{\partial
  f_{a+b}}{\partial x_0}\\
  \frac{\partial f_0}{\partial x_1} & \!\!\cdots\!\! & \frac{\partial
  f_{a+b}}{\partial x_1}\\
  \vdots & & \vdots\\
  \frac{\partial f_0}{\partial x_a} & \!\!\cdots\!\! & \frac{\partial
  f_{a+b}}{\partial x_a}\\
  \frac{\partial f_0}{\partial y_0} & \!\!\cdots\!\! & \frac{\partial
  f_{a+b}}{\partial y_0}\\
  \vdots & & \vdots\\
  \frac{\partial f_0}{\partial y_{b-1}} & \!\!\cdots\!\! & \frac{\partial
  f_{a+b}}{\partial y_{b-1}}
  \end{pmatrix}
  = \frac{1}{\prod_i x_i \prod_j y_j} \det\begin{pmatrix}
  x_0\frac{\partial f_0}{\partial x_0} & \!\!\!\cdots\!\! & x_0\frac{\partial
  f_{a+b}}{\partial x_0}\\
  x_1\frac{\partial f_0}{\partial x_1} & \!\!\!\cdots\!\! & x_1\frac{\partial
  f_{a+b}}{\partial x_1}\\
  \vdots & & \vdots\\
  x_a \frac{\partial f_0}{\partial x_a} & \!\!\!\cdots\!\! & x_a \frac{\partial
  f_{a+b}}{\partial x_a}\\
  y_0 \frac{\partial f_0}{\partial y_0} & \!\!\!\cdots\!\! & y_0 \frac{\partial
  f_{a+b}}{\partial y_0}\\
  \vdots & & \vdots\\
  y_{b-1}\frac{\partial f_0}{\partial y_{b-1}} & \!\!\!\cdots\!\! &
  y_{b-1}\frac{\partial
  f_{a+b}}{\partial y_{b-1}}
  \end{pmatrix}.
\]
For a fixed $\ell$, we have $f_\ell = x_i y_j$, which implies
\[
  f_\ell = x_i \frac{\partial f_\ell}{\partial x_i} = y_j \frac{\partial
  f_\ell}{\partial y_j},
\]
and all other partials vanish.  Hence the $\ell$th column is a
multiple of $f_\ell$, and once we factor out $f_\ell$, we are left
with the $\ell$th column of the truncated incidence matrix $M$.  Thus
\[
  J(f_0,\dots,f_{a+b}) = \frac{f_0 \cdots
  f_{a+b}}{\prod_i x_i \prod_j y_j} \ \det(M).
\]
The second part of the lemma is a standard consequence of the Matrix
Tree Theorem \cite[Chapter 12]{BM} which counts the number of spanning
trees of a graph.
\end{proof}

Now that we have the toric Jacobian, the next step in to introduce
duplicate sets of variables:
\[
  \mathbf{X} = (X_0,\dots,X_a),\ \mathbf{Y} = (Y_0,\dots,Y_b),\
  \mathbf{x} = (x_0,\dots,x_a),\ \mathbf{y} = (y_0,\dots,y_b).
\]
These give the polynomial ring
\[
  S \otimes S = k[\mathbf{X},\mathbf{Y},\mathbf{x},\mathbf{y}]
\]
and the ring homomorphism
\[
  S = k[\mathbf{x},\mathbf{y}] \longrightarrow S \otimes S
\]
defined by $x_i \mapsto X_i+x_i,y_i \mapsto Y_i+y_i$.  The image of $F
\in S$ in $S \otimes S$ is denoted $\widetilde{F}$, so that
\[
  \widetilde{F}(\mathbf{X},\mathbf{Y},\mathbf{x},\mathbf{y}) =
  F(\mathbf{X}+\mathbf{x},\mathbf{Y}+\mathbf{y}) \in S \otimes S.
\]
>From a canonical point of view, the map $F \mapsto \widetilde{F}$ is
comultiplication in the natural Hopf algebra structure on $S$.

The toric Jacobian $J$ gives a linear map
\[
  J : {\textstyle\bigwedge}^{a+b+1}W \longrightarrow  S_{b,a} \subset S
\]
and hence a map
\[
  \widetilde{J} : {\textstyle\bigwedge}^{a+b+1}W \longrightarrow  S
  \otimes S.
\]
Looking at homogeneous pieces, we have a decomposition
\[
  \widetilde{J} = {\textstyle\bigoplus}_{\alpha,\beta}
  J_{\alpha,\beta},
\]
where
\[
  J_{\alpha,\beta}: {\textstyle\bigwedge}^{a+b+1}W
  \longrightarrow S_{b-\alpha,a-\beta} \otimes S_{\alpha,\beta}
\]
lies in
\[
  \mathrm{Hom}_K({\textstyle\bigwedge}^{a+b+1}W, S_{b-\alpha,a-\beta}
  \otimes S_{\alpha,\beta}) \simeq
  \mathrm{Hom}_K({\textstyle\bigwedge}^{a+b+1}W \otimes
  S_{b-\alpha,a-\beta}^*,S_{\alpha,\beta}).
\]
Using \eqref{canonical}, $J_{\alpha,\beta}$ gives an element of
\[
  \mathrm{Hom}_E(\widehat{E}(a+b-p)\otimes
  S_{b-\alpha,a-\beta}^*,\widehat{E}(-p-1)\otimes S_{\alpha,\beta}),
\]
which by abuse of notation we write as
\begin{equation}
\label{jab2}
  J_{\alpha,\beta} : \widehat{E}(a+b-p)\otimes
  S_{b-\alpha,a-\beta}^* \longrightarrow \widehat{E}(-p-1)\otimes
  S_{\alpha,\beta}.
\end{equation}
In Section~\ref{mainthm} we will show that the map $d^p_{a+b,0}$ from
a Type 1 Tate resolution (see the discussion of following
Lemma~\ref{type1lemma}) can be chosen to be $J_{k+p+1,l+p+1}$.

\subsection{Type 2 and 3 Diagonal Maps}
\label{diagonal2}
The diagonal maps appearing the Type 2 and 3 Tate resolutions
discussed in Section~\ref{shape} are easy to describe.  We begin with
the map
\[
  \delta : {\textstyle\bigwedge}^{a+1}W \longrightarrow S_{0,a+1}
\]
defined as follows:\ given $f_0,\dots,f_a \in W$, we get the
\emph{Sylvester form}
\[
 \delta(f_0,\dots,f_a) = \det (\ell_{ij}), \ \text{where}\ f_i =
 {\textstyle\sum_{j=0}^a} \ell_{ij} x_j, \ \ell_{ij} \in S_{0,1}.
\]

Now fix $\alpha \ge 0$.  The multiplication map $S_{0,a+1} \otimes
S_{0,\alpha} \to S_{0,a+1+\alpha}$ induces
\[
  S_{0,a+1} \longrightarrow  S_{0,\alpha}^* \otimes S_{0,a+1+\alpha}
\]
and gives the composition
\[
  {\textstyle\bigwedge}^{a+1}W \xrightarrow{\ \delta \ } S_{0,a+1}
  \longrightarrow S_{0,\alpha}^* \otimes S_{0,a+1+\alpha}.
\]
This gives maps
\begin{align*}
  \delta_\alpha &: {\textstyle\bigwedge}^{a+1}W \otimes
  S_{0,\alpha} \longrightarrow S_{0,a+1+\alpha}\\ \delta_\alpha^*
  &: {\textstyle\bigwedge}^{a+1}W \otimes S_{0,a+1+\alpha}^*
  \longrightarrow S_{0,\alpha}^*
\end{align*}
and hence (by abuse of notation) maps
\begin{equation}
\label{deltatilde}
  \begin{aligned}
  \delta_\alpha &: \ed(a+1+k) \otimes
  S_{0,\alpha} \longrightarrow \ed(k)\otimes S_{0,a+1+\alpha}\\
  \delta_\alpha^*
  &: \ed(a+1+k) \otimes S_{0,a+1+\alpha}^*
  \longrightarrow \ed(k)\otimes S_{0,\alpha}^*.
  \end{aligned}
\end{equation}
In Section~\ref{mainthm} we will show that the diagonal map
$d^{p^-}_{a+b,b}$ from a Type 2 Tate resolution (see the discussion of
following Lemma~\ref{type2lemma}) and the map $d^{p^+ - 1}_{a,0}$ from
a Type 3 Tate resolution (see the discussion of following
Lemma~\ref{type3lemma}) can be chosen to be $\delta^*_{k-l-b-1}$ and
$\delta_{l-k-a-1}$ respectively.

We next consider the map
\[
  \delta' : {\textstyle\bigwedge}^{b+1}W \longrightarrow S_{b+1,0}
\]
defined as follows:\ given $f_0,\dots,f_b \in W$,
\[
  \delta'(f_0,\dots,f_a) = \det (\ell_{ij}'), \ \text{where}\
f_i = \sum_{j=0}^b \ell_{ij}' y_j, \ \ell_{ij}' \in S_{1,0}.
\]

As above, $\beta \ge 0$ gives the multiplication map $S_{b+1,0}
\otimes S_{\beta,0} \to S_{b+1+\beta,0}$ and the composition
\[
  {\textstyle\bigwedge}^{b+1}W \xrightarrow{\ \delta' \ } S_{b+1,0}
  \longrightarrow S_{\beta,0}^* \otimes S_{b+1+\beta,0}.
\]
This gives maps
\begin{equation}
\label{deltatildep}
  \begin{aligned}
  \delta_\beta' &:  \ed(b+1+l)\otimes S_{\beta,0}
  \longrightarrow \ed(l)\otimes S_{b+1+\beta,0}\\
  \delta_\beta^{\prime *} &:  \ed(b+1+l)\otimes
  S_{b+1+\beta,0}^* \longrightarrow \ed(l)\otimes S_{\beta,0}^*.
\end{aligned}
\end{equation}
In Section~\ref{mainthm} we will show that the map $d^{p^+ -1}_{b,0}$
from a Type 2 Tate resolution (see the discussion of following
Lemma~\ref{type2lemma}) and the map $d^{p^-}_{a+b,a}$ from a Type 3
Tate resolution (see the discussion of following
Lemma~\ref{type3lemma}) can be chosen to be $\delta'_{k-l-b-1}$ and
$\delta_{l-k-a-1}^{\prime *}$ respectively.

\subsection{The Main Theorem}
\label{mainthm}

Here is the main result of this section.

\begin{theorem}
For the Tate resolution $T^\bullet(\Fc)$ of the sheaf $\Fc =
\nu_*\Oc_X(k,l)$, the diagonal maps in $T^p(\Fc) \to T^{p+1}(\Fc)$ can
be chosen as follows:
\begin{enumerate}
\item {\rm (Type 1, $-a \le k-l \le b$):} $d^p_{a+b,0} = (-1)^p
  J_{k+p+1,l+p+1}$.
\item {\rm (Type 2, $k-l > b$):} $d^{p^-}_{a+b,b} =
  \delta^*_{k-l-b-1}$ and $d^{p^+ - 1}_{b,0} =
  \delta'_{k-l-b-1}$.
\item {\rm (Type 3, $k-l < -a$):} $d^{p^-}_{a+b,a} =
  \delta^{\prime *}_{l-k-a-1}$ and $d^{p^+ - 1}_{a,0} =
  \delta_{l-k-a-1}$.
\end{enumerate}
This uses the maps $J_{\alpha,\beta}, \delta_{\alpha},
\delta_{\alpha}^*, \delta'_{\beta}, \delta_{\beta}^{\prime *}$ defined
in \eqref{jab2}, \eqref{deltatilde} and \eqref{deltatildep}.
\end{theorem}

\begin{proof}
We begin with Type 2.  Let $\beta = k-l-b-1$ and assume $l = 0$ for
simplicity, so that $p^+ = 0$.  We will show that $T^{-2}(\Fc) \to
T^{-1}(\Fc) \to T^0(\Fc) \to T^{1}(\Fc)$ can be constructed as follows
using  $\delta'_{\beta}$:
\[
  \xymatrix@C=8pt@R=15pt{ \widehat{E}(b+2)\otimes S_{\beta-1,0}\otimes
  S_{0,1}^* \ar[r]^(.54){d^{-2}} & \widehat{E}(b+1)\otimes
  S_{\beta,0}\otimes S_{0,0}^* \ar[dr]^(.68){\delta'_{\beta}}\\ &
  & \hskip-48pt\widehat{E}(0)\otimes S_{\beta+b+1,0}
  \ar[r]^{\hskip-28pt d^0} & \widehat{E}(-1)\otimes S_{\beta+b+2,1}}.
\]
The differentials $d^{-2}$ and $d^0$ are the known horizontal maps.
To show that this sequence is exact, the first step is to prove that
$d^0\circ \delta'_{\beta} = \delta'_{\beta}\circ d^{-2} = 0$.
Consider the following identity that holds for all $f_0,\dots,f_{b+1}
\in W$:
\begin{equation}
\label{deltaidentity}
  \sum_{i = 0}^{b+1} (-1)^i f_i\, \delta'(f_0\wedge \cdots \widehat{f_i}
  \cdots \wedge f_{b+1}) = 0.
\end{equation}
If we write $f_i = \sum_{j=0}^b \ell_{ij}'y_j$, then
\eqref{deltaidentity} follows from the obvious identity
\[
  \det\begin{pmatrix} f_0 & \dots & f_{b+1}\\
  \ell_{0,0}' & \dots & \ell_{b+1,0}'\\
  \vdots && \vdots\\
  \ell_{0,b}' & \dots & \ell_{b+1,b}'
  \end{pmatrix} = 0
\]
by expanding by minors along the first row and using the definition of
$\delta'$.

By \eqref{canonical}, the composition
\[
  \widehat{E}(b+1)\otimes
  S_{\beta,0}\otimes S_{0,0}^* \xrightarrow{\delta'_{\beta}}
  \widehat{E}(0)\otimes S_{\beta+b+1,0} \xrightarrow{d^0}
  \widehat{E}(-1)\otimes S_{\beta+b+2,1}
\]
corresponds to a map
\[
  {\textstyle\bigwedge}^{b+2}W \otimes S_{\beta,0}\otimes S_{0,0}^*
  \longrightarrow S_{\beta+b+2,1}.
\]
We ignore $S_{0,0}^* \simeq k$.  Using the definition of
$\delta'_{\beta}$, this map is given by
\[
  f_0\wedge \cdots \wedge f_{b+1} \otimes h \longmapsto
  \sum_{i = 0}^{b+1} (-1)^i f_i\,h \, \delta'(f_0\wedge \cdots
  \widehat{f_i} \cdots \wedge f_{b+1})
\]
This reduces to zero (factor out $h \in S_{\beta,0}$ and use
\eqref{deltaidentity}), so $d^0\circ \delta'_{\beta} = 0$.

If $\beta > 0$, we need to consider $\delta'_{\beta} \circ
d^{-2}$.  Arguing as above, this map is determined by
\[
  {\textstyle\bigwedge}^{b+2}W \otimes S_{\beta-1,0}\otimes S_{0,1}^*
  \longrightarrow S_{\beta+b+1,0},
\]
which in turn is determined by the map
\[
  {\textstyle\bigwedge}^{b+2}W \otimes S_{\beta-1,0}
  \longrightarrow S_{\beta+b+1,0}\otimes S_{0,1} = S_{\beta+b+1,1}
\]
given by
\[
  f_0\wedge \cdots \wedge f_{b+1} \otimes h \longmapsto
  \sum_{i = 0}^{b+1} (-1)^i f_i\,h \, \delta'(f_0\wedge \cdots
  \widehat{f_i} \cdots \wedge f_{b+1})
\]
for $h \in S_{\beta-1,0}$.  As above, this reduces to zero, so that
$\delta'_{\beta} \circ d^{-2} = 0$.

When $\beta = 0$, we have to show that the composition
\[
  \xymatrix@C=10pt@R=10pt{ \ed(a+b+2)\otimes S^*_{0,a+1}
  \ar[dr]^{\delta_0^*} && \\ & \ed(a+1)\otimes S_{0,0}\otimes
  S_{0,0}^* \ar[dr]^(.6){\delta'_{0}}\\ & & \widehat{E}(0)\otimes
  S_{b+1,0}}
\]
is zero.  By \eqref{canonical}, the composed map corresponds to a map
\[
  {\textstyle\bigwedge}^{a+b+2}W \otimes S_{0,a+1}^* \longrightarrow
  S_{b+1,0},
\]
which in turn is determined by the a map
\[
  {\textstyle\bigwedge}^{a+b+2}W  \longrightarrow
  S_{b+1,0}\otimes S_{0,a+1} = S_{b+1,a+1}.
\]
Given $f_0,\dots,f_{a+b+1} \in W$, this map is given by
\begin{equation}
\label{beta0}
f_0 \wedge \cdots \wedge f_{a+b+1} \mapsto \sum_{|S| = a+1}
\varepsilon(S)\, \delta(\mathbf{f}_S)\, \delta'(\mathbf{f}_{S^c}),
\end{equation}
where the sum is over all subsets $S \subset \{0,\dots,a+b+1\}$ of
cardinality $a+1$ and $S^c = \{0,\dots,a+b+1\} \setminus S$.
Furthermore,
\begin{align*}
\delta(\mathbf{f}_S) &= \delta\big({\textstyle\bigwedge}_{i \in S}
f_i\big)\\
\delta'(\mathbf{f}_{S^c}) &= \delta'\big({\textstyle\bigwedge}_{i \in S^c}
f_i\big),
\end{align*}
and $\varepsilon(S) = \pm1$ is the sign that appears in the
Laplace expansion described below.

To show that the sum in \eqref{beta0} is zero, write $f_i =
\sum_{j=0}^a \ell_{ij}x_j = \sum_{j=0}^b \ell_{ij}'y_j$ and consider
the matrix
\[
\mathcal{M} = \begin{pmatrix}
\ell_{0,0} & \cdots & \ell_{a+b+1,0}\\
\vdots && \vdots\\
\ell_{0,a} & \cdots & \ell_{a+b+1,a}\\
\ell_{0,0}' & \cdots & \ell_{a+b+1,0}'\\
\vdots && \vdots\\
\ell_{0,b}' & \cdots & \ell_{a+b+1,b}'
\end{pmatrix}.
\]
If we multiply first $a+1$ rows by suitable $\mathbf{x}$ variables
and multiply the last $b+1$ rows by $\mathbf{y}$ variables, we get the
same result, namely the row $(f_0,\dots,f_{a+b+1})$.  If follows that
$\det \mathcal{M} = 0$.  If we take the Laplace expansion that
involves $(a+1)\times(a+1)$ minors of the first $a+1$ rows multiplied
by $(b+1)\times(b+1)$ complementary minors of the last $b+1$ rows, we
get the sum in \eqref{beta0}.  Hence this sum is zero, which proves
that $\delta'_{0} \circ \delta^*_{0} = 0$.

To complete the proof that $\delta'_{\beta}$ gives the diagonal
map in $T^{-1}(\Fc) \to T^0(\Fc)$, we follow the strategy used in
\cite[Thm.\ 1.3]{Cox_bez}.  Let $N' = (a+1)(b+1) = \dim(W)$.  Since
$\ed \simeq E(-N')$ and $T^{-1}(\Fc) \to T^0(\Fc) \to T^{1}(\Fc)$ is
\[
  \widehat{E}(b+1)\otimes
  S_{\beta,0}\otimes S_{0,0}^* \longrightarrow
  \widehat{E}(0)\otimes S_{\beta+b+1,0} \xrightarrow{d^0}
  \widehat{E}(-1)\otimes S_{\beta+b+2,1},
\]
the kernel of $d^0$ has $\dim(S_{\beta,0}\otimes S_{0,0}^*)$
minimal generators of degree $N'-b-1$.  Since we have proved that
$\delta'_\beta$ maps into this kernel, it suffices to prove that
this map is injective in degree $N'-b-1$, i.e., that
\[
  \delta'_\beta : {\textstyle\bigwedge}^{N'} W \otimes
  S_{\beta,0} \longrightarrow
  {\textstyle\bigwedge}^{N'-b-1} W \otimes S_{\beta+b+1,0}
\]
is injective (as above, we ignore $S_{0,0}^*$).  A basis of
${\textstyle\bigwedge}^{N'} W$ is given by $x_0y_0\wedge \cdots \wedge
x_0y_b \wedge \omega$, where $\omega$ is the wedge product of the
remaining $N'-b-1$ monomials of $W$ in some order.  Since
\[
  \delta'(x_0y_0\wedge \cdots \wedge x_0y_b) = x_0^{b+1},
\]
we see that for $h \in S_{\beta,0}$,
\[
  \delta'_\beta(x_0y_0\wedge \cdots \wedge x_0y_b \wedge \omega
  \otimes h) = \omega \otimes x_0^{b+1} h + \cdots \in
  {\textstyle\bigwedge}^{N'-b-1} W \otimes S_{\beta+b+1,0},
\]
where the omitted terms involve basis elements of $\bigwedge^{N'-b-1} W$
different from $\omega$.  The desired injectivity is now obvious.

This completes the proof for $d^{p^+ - 1}_{b,0}$ in a Type~2 Tate
resolution when $l = 0$ and $k = \beta + b + 1$.
The proof for arbitrary $l$ is similar, and
the same proof easily adapts to $d^{p^+ - 1}_{a,0}$ in a Type~3 Tate
resolution.  As for $d^{p^-}$, we note that applying
$\mathrm{Hom}_E(-,K)\otimes_K \ed$ to $T^p(\Fc)$ gives
$T^{a+b-p}(\Gc)$, where $\Gc = \nu_*\Oc_X(-a-1-k,-b-1-l)$.  This
duality interchanges Type~2 and Type~3 resolutions.  Then our results
for $d^{p^+ - 1}_{b,0}$ and $d^{p^+ - 1}_{a,0}$ and dualize to give
the desired results for $d^{p^-}_{a+b,a}$ and $d^{p^-}_{a+b,b}$.

It remains to consider Type 1 Tate resolutions.  This case will be
more complicated since there are two sets the variables to keep track
of:\ the original variables $\mathbf{x},\mathbf{y}$ and the duplicates
$\mathbf{X},\mathbf{Y}$ introduced in Section~\ref{diagonal1}.

Let $\alpha = k+p+1$ and $\beta = k+p+1$.  We will show that the
crucial part of $T^{p}(\Fc) \to T^{p+1}(\Fc) \to T^{p+2}(\Fc)$ can be
chosen to be
\[
  \xymatrix@C=13pt@R=10pt{
  \ed(a\!+\!b\!-\!p)\otimes S^*_{b\!-\!\alpha,a\!-\!\beta}
  \ar[r]^(.4){d^{p}_{a+b,a+b}}
  \ar[ddr]^{\ (-1)^p J_{\alpha,\beta}} & \ed(a\!+\!b\!-\!p\!-\!1)\otimes
  S^*_{b\!-\!\alpha\!-\!1,a\!-\!\beta\!-\!1}
  \ar[ddr]^{\ \ \ \ (-1)^{p+1} J_{\alpha+1,\beta+1}} & \\
  & {\textstyle\bigoplus} & \\
  &  \ed(-\!p\!-\!1)\otimes S_{\alpha,\beta} \ar[r]^(.47){d^{p+1}_{0,0}}
  & \widehat{E}(-\!p\!-\!2)\otimes S_{\alpha\!+\!1,\beta\!+\!1}}.
\]
This first step is to show that this is a complex, i.e., the
composition $T^{p}(\Fc) \to T^{p+1}(\Fc) \to T^{p+2}(\Fc)$ is zero.
Since the horizontal maps behave properly, it suffices to show that
\begin{equation}
\label{toshowtype1}
d^{p+1}_{0,0} \circ J_{\alpha,\beta} = J_{\alpha+1,\beta+1} \circ
d^{p}_{a+b,a+b}.
\end{equation}
Using \eqref{canonical}, this is equivalent to showing that the
diagram
\[
\xymatrix{ {\textstyle\bigwedge}^{a+b+2} W \otimes
  S^*_{b-\alpha,a-\beta} \ar[r]^(.45){d^{p}_{a+b,a+b}}
  \ar[d]^{J_{\alpha,\beta}}
  & {\textstyle\bigwedge}^{a+b+1}
  W \otimes S^*_{b-\alpha-1,a-\beta-1}
  \ar[d]^{J_{\alpha+1,\beta+1}} \\
  W  \otimes S_{\alpha,\beta} \ar[r]^{d^{p+1}_{0,0}}
  & S_{\alpha+1,\beta+1}}
\]
commutes.  A key point is that on the top, $d^{p}_{a+b,a+b}$ uses
$\mathbf{X},\mathbf{Y}$, while on the bottom, $d^{p+1}_{0,0}$ uses
$\mathbf{x},\mathbf{y}$.  We can recast the commutivity of this
diagram as saying that $d^{p+1}_{0,0} \circ J_{\alpha,\beta} =
J_{\alpha+1,\beta+1} \circ d^{p}_{a+b,a+b}$ as maps
\[
  {\textstyle\bigwedge}^{a+b+2} W \longrightarrow
  \underbrace{S_{b-\alpha,a-\beta}}_{\mathbf{X},\mathbf{Y}} \otimes
  \underbrace{S_{\alpha+1,\beta+1}}_{\mathbf{x},\mathbf{y}}.
\]

Given $a+b+2$ elements of $W$, we write them as $f_0,\dots,f_{a+b+1}$
when using $\mathbf{x},\mathbf{y}$ and as $F_0,\dots,F_{a+b+1}$
when using $\mathbf{X},\mathbf{Y}$.  Then \eqref{toshowtype1} is
equivalent to the identity
\[
\sum_{i=0}^{a+b+1}\!\! (-1)^i f_i J_{\alpha,\beta}(f_0\wedge \cdots\!
\widehat{f_i} \cdots \wedge f_{a+b+1})\!
=\!\! \sum_{i=0}^{a+b+1}\!\! (-1)^i F_i J_{\alpha+1,\beta+1}(f_0\wedge \cdots\!
\widehat{f_i} \cdots \wedge f_{a+b+1})
\]
in $S_{b-\alpha,a-\beta}\otimes S_{\alpha+1,\beta+1}$.  Summing this
over all $\alpha$ and $\beta$ gives the second identity
\[
\sum_{i=0}^{a+b+1} (-1)^i f_i\, \widetilde{J}(f_0\wedge\cdots
\widehat{f_i} \cdots \wedge f_{a+b+1})
= \sum_{i=0}^{a+b+1} (-1)^i F_i\, \widetilde{J}(f_0\wedge \cdots
\widehat{f_i} \cdots \wedge f_{a+b+1}),
\]
and the first identity follows from the second by taking the
appropriate graded piece.  However,
\begin{itemize}
\item The change of variables $(\mathbf{x},\mathbf{y}) \leftrightarrow
(\mathbf{X},\mathbf{Y})$ interchanges $f_i$ and $F_i$.
\item $\widetilde{J}(f_0\wedge \cdots \widehat{f_i} \cdots \wedge f_{a+b+1})$
is invariant under $(\mathbf{x},\mathbf{y}) \leftrightarrow
(\mathbf{X},\mathbf{Y})$.
\end{itemize}
It follows that the second identity is equivalent to the assertion that
\begin{equation}
\label{invariant}
\sum_{i=0}^{a+b+1} (-1)^i f_i\, \widetilde{J}(f_0\wedge \cdots
\widehat{f_i} \cdots \wedge f_{a+b+1})\ \text{is invariant under}\
(\mathbf{x},\mathbf{y}) \leftrightarrow (\mathbf{X},\mathbf{Y}).
\end{equation}
In particular, \eqref{toshowtype1} is an immediate consequence of
\eqref{invariant}.

We will prove \eqref{invariant} by representing $\sum_{i=0}^{a+b+1}
(-1)^i f_i\, \widetilde{J}(f_0\wedge \cdots \widehat{f_i} \cdots
\wedge f_{a+b+1})$ as a determinant.  We begin with the formula
\[
  J(f_0\wedge \cdots \wedge f_{a+b}) = \frac{1}{y_b} \det\begin{pmatrix}
  \frac{\partial f_0}{\partial x_0} & \!\!\cdots \!\!& \frac{\partial
  f_{a+b}}{\partial x_0}\\
  \vdots & & \vdots\\
  \frac{\partial f_0}{\partial x_a} & \!\!\cdots\!\! & \frac{\partial
  f_{a+b}}{\partial x_a}\\
  \frac{\partial f_0}{\partial y_0} & \!\!\cdots\!\! & \frac{\partial
  f_{a+b}}{\partial y_0}\\
  \vdots & & \vdots\\
  \frac{\partial f_0}{\partial y_{b-1}} & \!\!\cdots\!\! & \frac{\partial
  f_{a+b}}{\partial y_{b-1}}
  \end{pmatrix},
\]
which follows from the proof of Lemma~\ref{torictree}.  This implies
\[
  \widetilde{J}(f_0\wedge \cdots \wedge f_{a+b}) =
  \frac{1}{Y_b+y_b} \det\begin{pmatrix}
  \widetilde{\frac{\partial f_0}{\partial x_0}} & \!\!\cdots \!\!&
  \widetilde{\frac{\partial f_{a+b}}{\partial x_0}}\\
  \vdots & & \vdots\\
  \widetilde{\frac{\partial f_0}{\partial x_a}} & \!\!\cdots\!\! &
  \widetilde{\frac{\partial f_{a+b}}{\partial x_a}}\\
  \widetilde{\frac{\partial f_0}{\partial y_0}} & \!\!\cdots\!\! &
  \widetilde{\frac{\partial f_{a+b}}{\partial y_0}}\\
  \vdots & & \vdots\\
  \widetilde{\frac{\partial f_0}{\partial y_{b-1}}} & \!\!\cdots\!\! &
  \widetilde{\frac{\partial f_{a+b}}{\partial y_{b-1}}}
  \end{pmatrix}.
\]
It follows easily that
\[
\sum_{i=0}^{a+b+1}
(-1)^i f_i\, \widetilde{J}(f_0\wedge \cdots \widehat{f_i} \cdots
\wedge f_{a+b+1}) = \frac{1}{Y_b+y_b} \det \mathbf{M},
\]
where $\mathbf{M}$ is the $(a+b+2)\times(a+b+2)$ matrix
\[
\mathbf{M} = \begin{pmatrix} f_0 & \!\!\cdots \!\!& f_{a+b+1}\\
  \widetilde{\frac{\partial f_0}{\partial x_0}} & \!\!\cdots \!\!&
  \widetilde{\frac{\partial f_{a+b+1}}{\partial x_0}}\\
  \vdots & & \vdots\\
  \widetilde{\frac{\partial f_0}{\partial x_a}} & \!\!\cdots\!\! &
  \widetilde{\frac{\partial f_{a+b+1}}{\partial x_a}}\\
  \widetilde{\frac{\partial f_0}{\partial y_0}} & \!\!\cdots\!\! &
  \widetilde{\frac{\partial f_{a+b+1}}{\partial y_0}}\\
  \vdots & & \vdots\\
  \widetilde{\frac{\partial f_0}{\partial y_{b-1}}} & \!\!\cdots\!\! &
  \widetilde{\frac{\partial f_{a+b+1}}{\partial y_{b-1}}}
  \end{pmatrix}.
\]
To prove \eqref{invariant}, it suffices to show that $\det \mathbf{M}$
is unchanged when we replace its top row with $(F_0,\dots,F_{a+b+1})$.
For this purpose, consider the $(a+b+3)\times(a+b+3)$ matrix
\[
\overline{\mathbf{M}} = \begin{pmatrix}
&&&0\\
& \mathbf{M} && \vdots\\
&&&0\\
\widetilde{\frac{\partial f_0}{\partial y_{b}}} & \cdots &
  \widetilde{\frac{\partial f_{a+b+1}}{\partial y_{b}}} &1
\end{pmatrix}
\]
and observe that $\det \mathbf{M} = \det \overline{\mathbf{M}}$.
Write $\overline{\mathbf{M}}$ as
\[
\overline{\mathbf{M}} = \begin{pmatrix}
f_0 & \cdots & f_{a+b+1} & 0\\
&&&0\\
& \widetilde{Q} && \vdots\\
&&&1\\
\end{pmatrix}.
\]

Since $f_\ell \in W = S_{1,1}$, we have the easily proved identity
\[
F_\ell - f_\ell = - \sum_{i=0}^a x_i \widetilde{\frac{\partial
    f_\ell}{\partial x_{i}}} + \sum_{j=0}^b Y_j \widetilde{\frac{\partial
    f_\ell}{\partial y_{j}}}.
\]
Multiplying the last $a+b+2$ rows of $\overline{\mathbf{M}}$ by $-x_i$ or
$Y_j$ as appropriate and adding to the first row gives the matrix
\[
\overline{\mathbf{M}}' = \begin{pmatrix}
F_0& \cdots & F_{a+b+1} & Y_b\\
&&&0\\
& \widetilde{Q}  && \vdots\\
&&&1
\end{pmatrix}.
\]
Note that $\det\overline{\mathbf{M}}' = \det\overline{\mathbf{M}}$.
This is \emph{almost} what we need, except for the $Y_b$ in the first
row of $\overline{\mathbf{M}}'$.

We claim that $\det \widetilde{Q} = 0$.  Assuming this for the moment,
it follows that we can replace $Y_b$ with $0$ in
$\overline{\mathbf{M}}'$ without changing its determinant.  This
easily implies $\det \mathbf{M}$
is unchanged when we replace its top row with
$(F_0,\dots,F_{a+b+1})$ and will complete the proof of
\eqref{toshowtype1}.

It remains to study $\det \widetilde{Q}$.  The matrix $\widetilde{Q}$
is obtained from
\[
Q = \begin{pmatrix}
  \frac{\partial f_0}{\partial x_0} & \!\!\cdots \!\!&
  \frac{\partial f_{a+b}}{\partial x_0}\\
  \vdots & & \vdots\\
  \frac{\partial f_0}{\partial x_a} & \!\!\cdots\!\! &
  \frac{\partial f_{a+b}}{\partial x_a}\\
  \frac{\partial f_0}{\partial y_0} & \!\!\cdots\!\! &
  \frac{\partial f_{a+b}}{\partial y_0}\\
  \vdots & & \vdots\\
  \frac{\partial f_0}{\partial y_{b}} & \!\!\cdots\!\! &
  \frac{\partial f_{a+b}}{\partial y_{b}}
  \end{pmatrix}
\]
by the $F \mapsto \widetilde{F}$ operation described in
Section~\ref{diagonal1}.  But $\det Q = 0$ since $f_\ell =
\sum_{i=0}^a x_i \frac{\partial f_\ell}{\partial x_i} =
\sum_{j=0}^b y_j \frac{\partial f_\ell}{\partial y_j}$, and then
\[
\det \widetilde{Q} = \widetilde{\det Q} = 0.
\]

Hence we have proved that the maps $T^p(\Fc) \to T^{p+1}(\Fc)$ defined
using $(-1)^p J_{\alpha,\beta}$ give a complex.  To show that the
complex is exact, we again use the strategy of \cite[Thm.\
1.3]{Cox_bez}.  Lemma~\ref{type1lemma} tells us that $p^+ =
\min\{b-k,a-l\}$.  For simplicity, we assume $b-k \le a-l$, so that
$p^+ = b-k$.  Let $\beta = b-k+l$ and $p = p^+-1$.  Type 1 and $b-k
\le a-l$ imply $0 \le \beta \le a$.  Then $T^{p}(\Fc) \to
T^{p+1}(\Fc)$ becomes
\begin{equation}
\label{hardcase}
\begin{array}{c}  \xymatrix@C=20pt@R=8pt{
  \widehat{E}(a+k+1)\otimes
  S_{0,a-\beta}^* \ar[ddr]^{\ (-1)^{p}J_{b,\beta}} &  \\
  \bigoplus &  \\
  \widehat{E}(k-b+1) \otimes
  S_{b-1,\beta-1} \ar[r]^(.58){d^p_{0,0}} &
  \widehat{E}(k-b)\otimes
  S_{b,\beta}.
  }\end{array}
\end{equation}
Let $N' = (a+1)(b+1) = \dim(W)$.
Then the shape of the Tate resolution
tells us that there are $\dim(S_{b-1,\beta-1})$ minimal generators of
degree $N'-(k-b+1)$ and $\dim(S_{0,a-\beta}^*)$ minimal generators of
degree $N'-(a+k+1)$.  The former are taken care of by the known map
$d^p_{0,0}$, and for the latter, we see that in degree 
$N'-(a+k+1)$, the above diagram becomes
\begin{equation}
\label{equivar}
\begin{array}{c}
  \xymatrix@C=25pt@R=8pt{ {\textstyle\bigwedge}^{N'} W \otimes
  S_{0,a-\beta}^* \ar[ddr]^{\ (-1)^{p}J_{b,\beta}} & \\ \bigoplus & \\
  {\textstyle\bigwedge}^{N'-a-b} W \otimes S_{b-1,\beta-1}
  \ar[r]^{d^p_{0,0}} & {\textstyle\bigwedge}^{N'-a-b-1} W
  \otimes S_{b,\beta}.  }\end{array}
\end{equation}
As in \cite[Lem.\ 2.2]{Cox_bez}, we need to show that
$(-1)^{p}J_{b,\beta}$ is injective and that its image has trivial
intersection with the image of $d^p_{0,0}$.

For the former, let $\theta \in \bigwedge^{N'} W$ be the wedge product of
the monomials in $W$ in some order, and let $\varphi \in
S_{0,a-\beta}^*$ satisfy $J_{b,\beta}(\theta\otimes\varphi) = 0$.
Suppose that $\mathbf{Y}^u$ is a monomial in the $\mathbf{Y}$
variables of degree $|u| = a-\beta$.  We prove $\varphi(\mathbf{Y}^u)
= 0$ as follows.

Pick $\mathbf{Y}^v$ such that $\mathbf{Y}^u | \mathbf{Y}^v$ and $|v| =
a$, and write
\[
  \mathbf{Y}^v = Y_{j_1} \cdots Y_{j_a}.
\]
Then consider the following collection $f_0,\dots,f_{a+b}$ of
monomials in $W = S_{1,1}$:
\[
  x_0y_j,\ j = 0,\dots,b\ \ \text{and}\ \  x_i y_{j_i},\ i = 1,\dots,a.
\]
The graph of these monomials (in the sense of Section~\ref{diagonal1})
is easily seen to be a tree.  Then Lemma~\ref{torictree} implies that
\[
  J(f_0\wedge\cdots\wedge f_{a+b}) = \pm \frac{\prod_{j=0}^b x_0 y_j
  \prod_{i=1}^a x_i y_{j_i}}{\prod_{i=0}^a x_i \prod_{j=0}^b y_j} =
  \pm x_0^b\, {\textstyle\prod_{i=1}^a} y_{j_i} = \pm x_0^b\,
  \mathbf{y}^v.
\]
Thus $\widetilde{J}(f_0\wedge\cdots\wedge f_{a+b}) = \pm (X_0+x_0)^b
(\mathbf{Y}+\mathbf{y})^v$.  Taking those terms of degree $(b,\beta)$
in $(\mathbf{x},\mathbf{y})$, we obtain
\[
  J_{b,\beta}(f_0\wedge\cdots\wedge f_{a+b}) = \pm {\textstyle\sum_{w}}
  {\textstyle\binom{v}{w}} \, x_0^b \, \mathbf{Y}^{v-w} \,
  \mathbf{y}^w,
\]
where $\binom{v}{w} = \prod_{j=0}^b \binom{v_j}{w_j}$ and $\sum_w$
denotes the sum over all exponent vectors $w$ satisfying $|w| = \beta$
and $0 \le w_j \le v_j$ for all $j$.  Writing $\theta =
f_0\wedge\cdots\wedge f_{a+b} \wedge \omega$, we obtain
\begin{align*}
  0 &= J_{b,\beta}(f_0\wedge\cdots\wedge f_{a+b} \wedge \omega \otimes
  \varphi)\\
  &= \omega\otimes \varphi\big(J_{b,\beta}(f_0\wedge\cdots\wedge
  f_{a+b})\big)  + \cdots\\
  &= \omega\otimes \Big(\pm{\textstyle\sum_w \binom{v}{w}} \,
  \varphi\big(\mathbf{Y}^{v-w}\big) \, x_0^b \,\mathbf{y}^w\Big) +
  \cdots,
\end{align*}
where the omitted terms involve basis elements of $\bigwedge^{N'-a-b-1}
W$ different from $\omega$.  Since we are in characteristic $0$, it
follows that $\varphi(\mathbf{Y}^{v-w}) = 0$ for all $w$ under
consideration.  Our choice of $v$ guarantees that our original
monomial $\mathbf{Y}^u$ is one of these $\mathbf{Y}^{v-w}$'s.  Hence
$\varphi(\mathbf{Y}^{u}) = 0$, which implies $\varphi = 0$ since
$\mathbf{Y}^u$ was an arbitrary monomial of degree $a-\beta$.  This
completes the proof $(-1)^{p}J_{b,\beta}$ is injective.

It remains to show that the image of this map has trivial intersection
with the image of $d^p_{0,0}$. Following a suggestion of Jenia 
Tevelev, we use representation theory to finish the proof. 

Recall that there is a natural isomorphism $W = S_{1,1} \cong W_1
\otimes W_2$, where $W_1 = S_{1,0} = \C^{a+1}$ and $W_2 = S_{0,1} =
\C^{b+1}$.  First, we show that an action of the group $G =
\mathrm{SL}(W_1)\times \mathrm{SL}(W_2)$ on the diagram
(\ref{equivar}) is $G$-invariant on the maps $d_{0,0}^p$ and
$(-1)^{p}J_{b,\beta}$.  Indeed, since the map $d_{0,0}^p$ is induced
by the multiplication map
\[
  W\otimes S_{b-1,\beta-1}\rightarrow S_{b,\beta},
\]
we conclude that $d_{0,0}^p$ is $G$-invariant.  Now observe that the
toric Jacobian can be written as a linear combination of monomials
\begin{eqnarray*}
  J(f_0,\ldots,f_{a+b}) = \sum_{\mu,\nu}
  c_{\mu,\nu} x^{\mu} y^{\nu},
\end{eqnarray*}
where $c_{\mu,\nu}$ are the entries of the square matrix whose
determinant is a hyperdeterminant (see \cite[p.~473]{GKZ}).  By
\cite[Proposition~1.4]{Hyperdeterminants}, the hyperdeterminant is
$G$-invariant, so the toric Jacobian (\ref{toric_Jac}) (and
respectively the map $(-1)^p J_{b,\beta}$) is $G$-invariant.

It follows from Schur's Lemma that the images of $d_{0,0}^p$ and
$J_{b,\beta}$ have trivial intersection if the representation of $G$
corresponding to
\begin{equation}
\label{todecompose}
{\textstyle\bigwedge}^{N'-a-b} W \otimes S_{b-1,\beta-1} = 
{\textstyle\bigwedge}^{ab+1}(W_1\otimes W_2) \otimes
\textrm{Sym}^{b-1}(W_1)\otimes \textrm{Sym}^{\beta-1}(W_2)
\end{equation}
doesn't contain the representation corresponding to 
\[
{\textstyle\bigwedge}^{N'} W \otimes S_{0,a-\beta}^* = 
{\textstyle\bigwedge}^{ab+a+b+1}(W_1\otimes W_2) \otimes
\textrm{Sym}^{a-\beta}(W_2^*).
\]
To prove this, we use some basic facts from the representation theory
of the special linear group (see, e.g., \cite[\S6.1 and
\S15.3]{Fulton-Harris}).  Given a partition $\lambda = (\lambda_1,
\dots, \lambda_s)$, $\lambda_1 \ge \cdots \ge \lambda_s\ge 0$, we get
a \emph{Young diagram} $D_\lambda$, which consists of $s$ rows of
boxes, all starting at the same column, of lengths $\lambda_1 \ge
\cdots \ge \lambda_s$.

For a vector space $V$ over $K$, $\Sb_{\lambda}(V)$ denotes the
irreducible $\mathrm{SL}(V)$-representation corresponding to partition
$\lambda$.  We use notation $\lambda = (d_1^{a_1},\ldots, d_\ell^{a_\ell})$
to denote the partition having $a_i$ copies of the integer $d_i$,
$1\le i \le \ell$.  The corresponding Young diagram $D_\lambda$ has $a_i$
rows of boxes of length $d_i$.  Thus $\lambda = (d)$ gives the
symmetric product $\Sb_{\lambda}(V) = \textrm{Sym}^d(V)$ and $\lambda =
(1^d)$ gives the exterior product $\Sb_{\lambda}(V) = \bigwedge^d V$.  

Recall that $\Sb_{\lambda}(V) = 0$ when the Young diagram of $\lambda$
has more than $\dim V$ nonzero rows, and that two Young diagrams give
the same $\mathrm{SL}(V)$-representation if and only if one can be
obtained from the other by adding or deleting columns of height $\dim
V$ at the beginning of the Young diagram.  

By the Cauchy formula \cite[\S6.1]{Fulton-Harris}, we have the
following decomposition for the exterior powers of $W = W_1 \otimes
W_2$:
\begin{eqnarray*}
  {\textstyle\bigwedge}^{ab + 1} W = 
  {\textstyle\bigwedge}^{ab + 1}(W_1\otimes W_2) = 
  \bigoplus_{|\lambda| = ab + 1} \Sb_{\lambda}(W_1)\otimes
  \Sb_{\lambda'}(W_2),
\end{eqnarray*}
where the direct sum runs over all partitions $\lambda$ of $ab + 1$ with at
most $\dim W_1 = a+1$ rows, at most $\dim W_2 = b+1$ columns, and
$\lambda'$ is the conjugate partition to $\lambda$.  Note that the
representation corresponding to the highest power of determinant
${\textstyle\bigwedge}^{N'} W$ is one-dimensional, i.e., is a trivial
representation.

When we combine this with \eqref{todecompose}, we see that it is
enough to show that it cannot happen simultaneously that
$\Sb_{\lambda}(W_1)\otimes \textrm{Sym}^{b-1}(W_1)$ contains the
trivial representation and $\Sb_{\lambda'}(W_2)\otimes
\textrm{Sym}^{\beta-1}(W_2)$ contains $\textrm{Sym}^{a-\beta}(W_2^*)$.
Since $\dim W_1 = a+1$ and $\dim W_2 = b+1$, we can assume that the
Young diagram of $\lambda$ has at most $a+1$ rows (otherwise
$\Sb_{\lambda}(W_1) = 0$) and at most $b+1$ columns (otherwise
$\Sb_{\lambda'}(W_2) =0$).

By the Pieri formula \cite[(6.8)]{Fulton-Harris}, for any partition
$\lambda$, we have
\begin{eqnarray*}
  \Sb_{\lambda}(W_1)\otimes \mathrm{Sym}^{b-1}(W_1) 
  \cong \bigoplus_{\nu}\Sb_{\nu}(W_1),
\end{eqnarray*}
where the sum is over all $\nu$ whose Young diagram is obtained by
adding $b-1$ boxes to the Young diagram of $\lambda$, with no two
boxes in the same column.  Note also that each $\nu$ is a partition
of $(ab + 1) + (b-1) = (a+1)b$.  Since $D_\lambda$ has $|\lambda| =
ab+1$ boxes and fits inside a $(a+1)\times(b+1)$ rectangle, the only
way for $\nu$ to give the trivial representation is for $D_\lambda$ to
be the Young diagram:
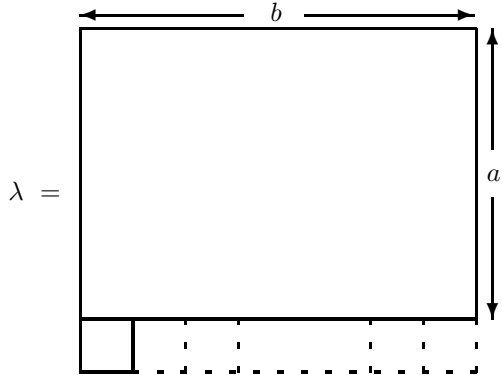
\[
\begin{picture}(200,150)
\put(3,75){$\lambda \ = $}
\thicklines
\put(30,140){\line(1,0){150}}
\put(30,30){\line(1,0){150}}
\put(180,30){\line(0,1){110}}
\put(30,10){\line(0,1){130}}
\put(30,10){\line(1,0){20}}
\put(50,10){\line(0,1){20}}
\put(102,143){$b$}
\put(115,145){\vector(1,0){64}}
\put(92,145){\vector(-1,0){62}}
\put(184,82.5){$a$}
\put(186,76){\vector(0,-1){46}}
\put(186,94){\vector(0,1){46}}
\dashline{3}(50,10)(180,10)(180,30)
\dashline{3}(70,10)(70,30)
\dashline{3}(90,10)(90,30)
\dashline{3}(160,10)(160,30)
\dashline{3}(140,10)(140,30)
\end{picture}
\]
You can see how adding $b-1$ boxes to the bottom row (the dashed boxes
in the drawing) give the trivial representation, since $D_\nu$ is
trivial if and only if it consists entirely of columns of height $a+1$.

This shows that the only case when $\Sb_{\lambda}(W_1)\otimes
\textrm{Sym}^{b-1}(W_1)$ contains the trivial representation is when
$\lambda = (b^a,1)$.  Hence, $\lambda'$ must be $(a+1,a^{b-1})$.  On
the other hand, $\textrm{Sym}^{a-\beta}(W_2^*)$ corresponds to the
partition $(b^{a-\beta})$ (see \cite[\S15.5, Exercise
15.50]{Fulton-Harris}), so from the Pieri formula we see that that it
is impossible to get $(b^{a-\beta})$ from the tensor product
$\Sb_{\lambda'}(W_2)\otimes \textrm{Sym}^{\beta-1}(W_2)$ by adding
$\beta-1$ boxes to $(a+1,a^{b-1})$, no two in the same column, and
then deleting columns of height $b+1$.

The final step is to prove exactness when $T^p \to T^{p+1}$ is given by
\[
\xymatrix@C=20pt@R=8pt{
  \widehat{E}(a+k+1)\otimes
  S_{b-\alpha,a-\beta}^* \ar[ddr]^{\ (-1)^{p}J_{\alpha,\beta}} 
  \ar[r]^(.465){d^p_{a+b,a+b}} &  
  \widehat{E}(a+k) \otimes S_{b-\alpha-1,a-\beta-1}^*\\
  \bigoplus & \bigoplus \\
  \widehat{E}(k-b+1) \otimes
  S_{\alpha-1,\beta-1} \ar[r]^(.58){d^p_{0,0}} &
  \widehat{E}(k-b)\otimes
  S_{\alpha,\beta}.
  }
\]
Here, we use the same conventions as in \eqref{hardcase}, except that
we now assume that $b-\alpha$ and $a-\beta$ are positive.  As before,
the shape of the Tate resolution tells us that there are
$\dim(S_{\alpha-1,\beta-1})$ minimal generators of degree $N'-(k-b+1)$
and $\dim(S_{b-\alpha,a-\beta}^*)$ minimal generators of degree
$N'-(a+k+1)$.  The former are taken care of by the known map
$d^p_{0,0}$, and for the latter, we see that in degree $N'-(a+k+1)$,
the above diagram becomes
\begin{equation}
\label{easycase}
\begin{array}{c}
  \xymatrix@C=25pt@R=8pt{ {\textstyle\bigwedge}^{N'} W \otimes
  S_{b-\alpha,a-\beta}^* \ar[ddr]^{\ (-1)^{p}J_{\alpha,\beta}} 
  \ar[r]^(.465){d^p_{a+b,a+b}} & 
   {\textstyle\bigwedge}^{N'-1}W\otimes S_{b-\alpha-1,a-\beta-1}^*\\ 
  \bigoplus & \bigoplus \\
  {\textstyle\bigwedge}^{N'-a-b} W \otimes S_{\alpha-1,\beta-1}
  \ar[r]^{d^p_{0,0}} & {\textstyle\bigwedge}^{N'-a-b-1} W
  \otimes S_{\alpha,\beta}.  }\end{array}
\end{equation}
The map $d^p_{a+b,a+b}$ is injective since it is dual to the
surjective multiplication map $W \otimes S_{b-\alpha-1,a-\beta-1} \to
S_{b-\alpha,a-\beta}$.  As in the proof of \cite[Thm.\ 1.3]{Cox_bez},
it follows immediately that the map \eqref{easycase} is injective on
$\bigwedge^{N'} W \otimes S_{b-\alpha,a-\beta}^*$ and that the images
of $\bigwedge^{N'} W \otimes S_{b-\alpha,a-\beta}^*$ and
$\bigwedge^{N'-a-b} W \otimes S_{\alpha-1,\beta-1}$ have trivial
intersection.  This completes the proof of the theorem.
\end{proof}

\begin{remark}
In the proof of Theorem~\ref{mainthm}, we used the relation between
the toric Jacobian of $f_0,\dots,f_{a+b} \in S_{1,1}$ and the
hyperdeterminants studied in \cite{GKZ,Hyperdeterminants} prove the
equivariance we needed.  The theorem implies that certain
hyperdeterminants are explicitly encoded into the Tate resolutions
resolutions considered here.  This is another example of the amazing
amount of information contained in these resolutions.
\end{remark}


\section*{Acknowledgements}

E. Materov was partially supported by the Russian Foundation for Basic
Research, grant 05-01-00517, by grant 06-01-91063 from the Japanese
Society for the Promotion of Science and the Russian Foundation for
Basic Research, and by NM project 45.2007 of the Siberian Federal
University grant.  We are very grateful to Jenia Tevelev for
discussions about Young diagrams and representation theory.  We would
like to thank Rob Benedetto for a helpful suggestion in the proof of
Theorem~\ref{mainthm}.

\end{document}